\documentclass[11pt,xcolor=pdftex,a4paper,oneside,centertags,reqno,endnotes]{amsart}

\usepackage{endnotes}
\usepackage{latexsym}
\usepackage{amsmath}
\usepackage{amsfonts}
\usepackage{amssymb}   
\usepackage{epsfig}
\usepackage{amscd}
\usepackage{amsthm}
\usepackage{enumerate}
\usepackage{eucal}
\usepackage[utf8x]{inputenc}
\usepackage{graphics}

\usepackage[normalem]{ulem}

\newcommand\rmi{\hbox{\rm (i)}}
\newcommand\rmii{\hbox{\rm (ii)}}
\newcommand\rmiii{\hbox{\rm (iii)}}

\usepackage{pdfsync}
\usepackage{color}
\usepackage{mathrsfs}

\usepackage[margin=1.3in]{geometry}

\usepackage[T1]{fontenc}
\usepackage{mathrsfs}
\usepackage{epsfig}
\usepackage{xypic}
\usepackage{verbatim}
\usepackage{multicol}
\usepackage{wrapfig}
\usepackage{pgf,tikz}
\usepackage{mathrsfs}
\usetikzlibrary{arrows}

\numberwithin{equation}{section}

\newtheorem{theorem}{Theorem}[section]
\newtheorem{proposition}[theorem]{Proposition}
\newtheorem{lemma}[theorem]{Lemma}
\newtheorem{corollary}[theorem]{Corollary}
\theoremstyle{definition}
\newtheorem{definition}[theorem]{Definition}
\newtheorem{remark}[theorem]{Remark}
\newtheorem{notation}[theorem]{Notation}
\newtheorem*{notation*}{Notation}

\renewcommand{\Re}{\mathrm{Re}}

\newcommand{\divv}{\mathrm{div}}

\newcommand{\Dom}{\mathrm{Dom}}

\newcommand{\sign}{\mathrm{sign}}

\newcommand{\N}{\mathbb{N}}
\newcommand{\Z}{\mathbb{Z}}
\newcommand{\R}{\mathbb{R}}
\newcommand{\C}{\mathbb{C}}
\newcommand{\mc}[1]{\mathcal{#1}}
\newcommand{\naH}{\nabla_\mathcal{H}}

\usepackage{textcomp}
\usepackage{dsfont}
\usepackage{latexsym}
\usepackage{amssymb}
\usepackage{amsthm}
\usepackage{amsmath}
\DeclareMathAlphabet{\mathpzc}{OT1}{pzc}{m}{en}
\usepackage{yfonts}
\usepackage{xfrac}
\usepackage{fge}

\usepackage{fancyhdr}

\usepackage{newlfont}
\usepackage{graphicx}

\usepackage{mathtools}
\usepackage{comment}
\usepackage{indentfirst}
\usepackage{braket}

\DeclareMathOperator{\supp}{supp}

\DeclarePairedDelimiter{\abs}{\lvert}{\rvert}

\DeclarePairedDelimiter{\norm}{\lVert}{\rVert}

{\left\lbrace\begin{array}{@{}l@{}}}%
{\end{array}\right.}

\def\Supp#1{\supp\left( #1\right) }

\makeatletter
\newcommand*{\mint}[1]{%
	\mint@l{#1}{}%
}
\newcommand*{\mint@l}[2]{%
	\@ifnextchar\limits{%
		\mint@l{#1}%
	}{%
	\@ifnextchar\nolimits{%
		\mint@l{#1}%
	}{%
	\@ifnextchar\displaylimits{%
		\mint@l{#1}%
	}{%
	\mint@s{#2}{#1}%
}%
}%
}%
}
\newcommand*{\mint@s}[2]{%
	\@ifnextchar_{%
		\mint@sub{#1}{#2}%
	}{%
	\@ifnextchar^{%
		\mint@sup{#1}{#2}%
	}{%
	\mint@{#1}{#2}{}{}%
}%
}%
}
\def\mint@sub#1#2_#3{%
	\@ifnextchar^{%
		\mint@sub@sup{#1}{#2}{#3}%
	}{%
	\mint@{#1}{#2}{#3}{}%
}%
}
\def\mint@sup#1#2^#3{%
	\@ifnextchar_{%
		\mint@sub@sup{#1}{#2}{#3}%
	}{%
	\mint@{#1}{#2}{}{#3}%
}%
}
\def\mint@sub@sup#1#2#3^#4{%
	\mint@{#1}{#2}{#3}{#4}%
}
\def\mint@sup@sub#1#2#3_#4{%
	\mint@{#1}{#2}{#4}{#3}%
}
\newcommand*{\mint@}[4]{%
	\mathop{}%
	\mkern-\thinmuskip
	\mathchoice{%
		\mint@@{#1}{#2}{#3}{#4}%
		\displaystyle\textstyle\scriptstyle
	}{%
	\mint@@{#1}{#2}{#3}{#4}%
	\textstyle\scriptstyle\scriptstyle
}{%
\mint@@{#1}{#2}{#3}{#4}%
\scriptstyle\scriptscriptstyle\scriptscriptstyle
}{%
\mint@@{#1}{#2}{#3}{#4}%
\scriptscriptstyle\scriptscriptstyle\scriptscriptstyle
}%
\mkern-\thinmuskip
\int#1%
\ifx\\#3\\\else_{#3}\fi
\ifx\\#4\\\else^{#4}\fi  
}
\newcommand*{\mint@@}[7]{%
	\begingroup
	\sbox0{$#5\int\m@th$}%
	\sbox2{$#5\int_{}\m@th$}%
	\dimen2=\wd0 %
	\let\mint@limits=#1\relax
	\ifx\mint@limits\relax
	\sbox4{$#5\int_{\kern1sp}^{\kern1sp}\m@th$}%
	\ifdim\wd4>\wd2 %
	\let\mint@limits=\nolimits
	\else
	\let\mint@limits=\limits
	\fi
	\fi
	\ifx\mint@limits\displaylimits
	\ifx#5\displaystyle
	\let\mint@limits=\limits
	\fi
	\fi
	\ifx\mint@limits\limits
	\sbox0{$#7#3\m@th$}%
	\sbox2{$#7#4\m@th$}%
	\ifdim\wd0>\dimen2 %
	\dimen2=\wd0 %
	\fi
	\ifdim\wd2>\dimen2 %
	\dimen2=\wd2 %
	\fi
	\fi
	\rlap{%
		$#5%
		\vcenter{%
			\hbox to\dimen2{%
				\hss
				$#6{#2}\m@th$%
				\hss
			}%
		}%
		$%
	}%
	\endgroup
}

\makeatletter
\def\overbracket#1{\mathop{\vbox{\ialign{##\crcr\noalign{\kern3\p@}
				\downbracketfill\crcr\noalign{\kern3\p@\nointerlineskip}
				$\hfil\displaystyle{#1}\hfil$\crcr}}}\limits}
\def\underbracket#1{\mathop{\vtop{\ialign{##\crcr
				$\hfil\displaystyle{#1}\hfil$\crcr\noalign{\kern3\p@\nointerlineskip}
				\upbracketfill\crcr\noalign{\kern3\p@}}}}\limits}
\def\overparenthesis#1{\mathop{\vbox{\ialign{##\crcr\noalign{\kern3\p@}
				\downparenthfill\crcr\noalign{\kern3\p@\nointerlineskip}
				$\hfil\displaystyle{#1}\hfil$\crcr}}}\limits}
\def\underparenthesis#1{\mathop{\vtop{\ialign{##\crcr
				$\hfil\displaystyle{#1}\hfil$\crcr\noalign{\kern3\p@\nointerlineskip}
				\upparenthfill\crcr\noalign{\kern3\p@}}}}\limits}
\def\downparenthfill{$\m@th\braceld\leaders\vrule\hfill\bracerd$}
\def\upparenthfill{$\m@th\bracelu\leaders\vrule\hfill\braceru$}
\def\upbracketfill{$\m@th\makesm@sh{\llap{\vrule\@height3\p@\@width.7\p@}}%
	\leaders\vrule\@height.7\p@\hfill
	\makesm@sh{\rlap{\vrule\@height3\p@\@width.7\p@}}$}
\def\downbracketfill{$\m@th
	\makesm@sh{\llap{\vrule\@height.7\p@\@depth2.3\p@\@width.7\p@}}%
	\leaders\vrule\@height.7\p@\hfill
	\makesm@sh{\rlap{\vrule\@height.7\p@\@depth2.3\p@\@width.7\p@}}$}
\makeatother

\newcommand{\loc}{\mathrm{loc}}

\newcommand{\dd}{d}

\newcommand{\Hc}{\mathcal{H}}

\begin{document}
\title[Weighted sub-Laplacians on Métivier Groups]{Weighted sub-Laplacians on Métivier Groups: Essential Self-Adjointness and Spectrum}
\author[Bruno]{Tommaso Bruno}
\author[Calzi]{Mattia Calzi}

\address{Tommaso Bruno \\Università degli Studi di Genova\\ Dipartimento di Matematica\\ Via Dodecaneso\\ 35 16146 Genova\\ Italy }
\email{brunot@dima.unige.it}

\address{Mattia Calzi \\Scuola Normale Superiore\\ Piazza dei Cavalieri \\ 7 56126 Pisa\\ Italy }
\email{mattia.calzi@sns.it}

\maketitle
\begin{small}
	\section*{Abstract}
	Let $G$ be a Métivier group and let $N$ be any homogeneous norm on $G$. For $\alpha>0$ denote by $w_\alpha$ the function $ e^{-N^\alpha}$ and consider the weighted sub-Laplacian $\mathcal{L}^{w_\alpha}$ associated with the Dirichlet form $\phi \mapsto \int_{G} \abs{\nabla_\mathcal{H}\phi(y)} ^2 w_\alpha(y)\, \dd y$, where $\nabla_\mathcal{H}$ is the horizontal gradient on $G$. Consider $\mathcal{L}^{w_\alpha}$ with domain $C_c^\infty$. We prove that $\mathcal{L}^{w_\alpha}$ is essentially self-adjoint when $\alpha \geq 1$. For a particular $N$, which is the norm appearing in $\mathcal{L}$'s fundamental solution when $G$ is an H-type group, we prove that $\mathcal{L}^{w_\alpha}$ has purely discrete spectrum if and only if $\alpha>2$, thus proving a conjecture of J.\ Inglis.
\end{small}

\section{Introduction and preliminary results}\label{sec:intr}
Let $G$ be a stratified group. Denote by $dy$ its left- and right-invariant Haar measure and by $\nabla_\mathcal{H}$ the horizontal gradient on $G$. Given a positive measurable function $w$ on $G$, denote by $L^2(G, w) \eqqcolon L^2(w)$ the space of all functions on $G$ that are square integrable with respect to the measure $w(y) \, \dd y$, with consequent inner product. We define the weighted sub-Laplacian $\mathcal{L}^w$ as the operator associated with the Dirichlet form
\[\phi \mapsto \int_G \abs{\nabla_\mathcal{H} \phi(y)}^2 w(y)\, \dd y\]
for every $\phi \in C_c^\infty$. Under very weak assumptions on $w$, such form is symmetric, densely defined, accretive and continuous; thus the operator $\mathcal{L}^w$, with domain $C_c^\infty$, has self-adjoint extensions. We consider in particular the family of weights
\[(w_\alpha = e^{-N^\alpha })_{\alpha>0}\]
where $N$ is any homogeneous norm on $G$. When $G$ is some $\R^n$, $\alpha=2$ and the homogeneous norm is the Euclidean norm on $\R^n$, this operator is indeed -- up to normalization constants -- the classical 
Ornstein-Uhlenbeck operator; therefore $\mathcal{L}^{w_2}$, for some choice on $N$, can be considered as one of its possible sub-Riemannian versions on stratified groups. $\mathcal{L}^{w_\alpha}$ for $\alpha>0$ is then its natural generalization.

The aim of this paper is to investigate the essential self-adjointness of $\mathcal{L}^{w_\alpha}$ and, in the particular case when $G$ is a Métivier group, whether its self-adjoint extensions have purely discrete spectra. We recall that a self-adjoint operator $T$ on a Hilbert space has purely discrete spectrum if its spectrum $\sigma(T)$ is discrete, and for each $\lambda \in \sigma(T)$ the range of the spectral projection $E(\{\lambda\})$ is finite-dimensional (cf.~\cite{ReedSimonI}). When $G$ is an H-type group (a subclass of M\'etivier groups) 
and $N$ is either the Carnot-Charath\'eodory norm or the Kaplan 
norm, the discretness of the spectrum of $\mathcal{L}^{w_\alpha}$ 
has been investigated by J.~Inglis \cite{Inglis}. By using super 
Poincar\'e inequalities he proved that when $N$ is the Carnot-Charath\'eodory norm $\mathcal{L}^{w_\alpha}$ has discrete 
spectrum for  every $\alpha>1$. When $N$ is the Kaplan norm he 
proved that the spectrum of $\mathcal{L}^{w_\alpha}$ is non-discrete for $1<\alpha<2$. Even though his technique, based on functional inequalities, did not allow him to investigate the spectrum 
of $\mathcal{L}^{w_\alpha}$ for the Kaplan norm when $\alpha\ge 2$, Inglis conjectured 
that  the spectrum 
is discrete when $\alpha>2$.
 By using different techniques, we completely characterize the 
discreteness of the spectrum of $\mathcal{L}^{w_\alpha}$ for every 
$\alpha>0$ in the case of the Kaplan norm, not only solving Inglis's 
conjecture, but also giving a direct proof of the non-discreteness of 
the spectrum when $0<\alpha \leq 2$, thus improving Inglis's result.

The content of this paper is as follows. In the remaining of this section we fix the notation and recall the basic features of analysis on stratified groups. In Section~\ref{sec:weighted:schr} we prove a general criterion (Theorem~\ref{teo}) for establishing existence and uniqueness of self-adjoint extensions of $\mathcal{L}^w$, with domain $C_c^\infty$, on $L^2(w)$. It is based on suitable integrability or (essential) boundedness properties of the weight $w$ and  its horizontal derivatives. As a corollary, the essential self-adjointness of $\mathcal{L}^{w_\alpha}$ when $\alpha \geq 1$ (Theorem~\ref{essselfadj:walpha}) is obtained. By using a well known technique (see e.g.~\cite{DaviesSimon, Carbonaro, Inglis}), in Section~\ref{SchrodingerSection} we then show how to reduce the study of $\mathcal{L}^w$ to the study of the Schr\"odinger operator $\mathcal{L} + V$ as an operator on $L^2(G, \dd y)$, where the potential $V$ is
\begin{equation*}
V= -\frac{1}{4}\frac{\abs*{\nabla_\mathcal{H}w}^2}{w^2} - \frac{1}{2}\frac{\mathcal{L}w}{w}.
\end{equation*}

In Section~\ref{sec:discr}, we restrict our analysis to M\'etivier groups and the case when $N$ is the Kaplan  norm. We prove that the spectrum of $\mathcal{L}^{w_\alpha}$ is purely discrete if and only if $\alpha>2$, thus proving Inglis's conjecture \cite{Inglis}. To prove the theorem we reduce matters to studying the spectrum of the Schr\"odinger operator $\mathcal{L} + V_\alpha$
obtained by conjugating $\mathcal{L}^{w_\alpha}$ with a suitable isometry. We then apply a generalization to stratified groups of a theorem of B.~Simon relating the discreteness of the spectrum to the behaviour of the sublevel sets of the potential $V_\alpha$ \cite{Simon1}.
\subsection{Preliminaries.}
A stratified group $G$ is a connected, simply connected Lie group whose Lie algebra $\mathfrak{g}$ admits a direct sum decomposition
\begin{equation}\label{stratification}
\mathfrak{g}=\mathfrak{g}_1 \oplus \cdots \oplus \mathfrak{g}_\nu
\end{equation}
such that $[\mathfrak{g}_1,\mathfrak{g}_{k-1}]= \mathfrak{g}_k$ for every $k\leq \nu$, and $[\mathfrak{g}_1, \mathfrak{g}_\nu]=0$. The first layer $\mathfrak{g}_1$ is often referred to as \emph{horizontal layer}. We set $n_k = \dim \mathfrak{g}_k$; $Q=\sum_{k=1}^\nu k \, n_k$ will denote the homogeneous dimension of $G$. We shall write $G^*$ instead of $G\setminus\{e\}$.

The Lie algebra $\mathfrak{g}$ is canonically endowed with the family of dilations $\delta_r \left(\sum_{k=1}^\nu Y_k\right) = \sum_{k=1}^{\nu} r^k Y_k$, where $Y_k\in \mathfrak{g}_k$. One can then define a family of dilations on $G$ via the exponential map, by $\delta_r (\exp Y) = \exp \delta_r(Y)$.

Recall that $G$ is unimodular, so that it can be endowed with a Haar measure $\dd y$ which is both left and right invariant. Thus, one can define the space $L^p$ as the space of (equivalence classes of) measurable functions $f$ such that $\abs{f}^p$ is integrable with respect to $\dd y$, with the usual norm.
The space $L^p(\mathfrak{g}_1)$ of $p$-th power integrable $\mathfrak{g}_1$-valued functions is defined analogously.

From now on, we fix a basis $(X_j)_{1\leq j\leq n_1}$ of left-invariant vector fields for $\mathfrak{g}_1$. We endow $\mathfrak{g}_1$ with the unique inner product which turns $(X_j)$ into an orthonormal basis; if $f$ is a differentiable function on $G$, its horizontal gradient is the vector field
\[
\nabla_\mathcal{H} f \coloneqq \sum_{j=1}^{n_1} (X_j f)X_j.
\]
Observe that, if $f,g$ are smooth functions on $G$, then $(\nabla_\Hc f)(g)=\nabla_\Hc f \cdot \nabla_\Hc g$, where the right hand side is the inner product of the vector fields $\nabla_\Hc f, \nabla_\Hc g$. 

For a vector-valued differentiable function $F=(F_1,\dots,F_{n_1})$, its horizontal divergence is the function
\[\divv_\Hc F = \sum_{j=1}^{n_1}X_j F_j,\]
and the sub-Laplacian $\mathcal{L}$ will then be the (positive) operator
\[
\mathcal{L}\coloneqq - \divv_\Hc \nabla_\Hc = -\sum_{j=1}^{n_1} X_j^2.
\] 
With domain $C_c^\infty$, $\mathcal{L}$ is well known to be essentially self-adjoint on $L^2$. With a slight abuse of notation, we shall denote by $\mathcal{L}$ also its closure and the operator interpreted in the sense of distributions. $\Dom(\mathcal{L})$ will stand for the set of $f\in L^2$ such that $\mathcal{L} f$ also belongs to $L^2$.

\begin{definition}[cf.~\cite{Folland}]\label{defSob}
	For every $k\in \Z$ define $S^{2}_{k}$ as the completion of $C_c^\infty$ with respect to the norm
	\[f \mapsto \norm{(I+\mathcal{L})^{k/2}f}_{L^2},\] 
	and
	\[S^{2}_{k,\loc}= \{f\in \mathcal{D}'\colon \; \phi f \in S^{2}_{k} \;\; \forall \phi \in C_c^\infty\}.\]
\end{definition}

In the following lemma we state some general properties of the Sobolev spaces just introduced, which will be used later on. They are the analogues for $S^2_{k,\loc}$ of some properties of $S^2_k$ that can be found in~\cite{Folland}. The proof is omitted.

\begin{lemma}\label{lemmainiz}
	The following hold:
	\begin{itemize}
		\item[\rmi] $S^{2}_{1,\loc} = \{f\in L^2_\loc \colon \nabla_\Hc f\in L^2_\loc\}$ and $S^{2}_{2,\loc} = \{f\in S^{2}_{1,\loc} \colon \mathcal{L} f\in L^2_\loc\}$;
		\item[\rmii] If $f\in L^2_\loc$, then $X f \in S^{2}_{-1,\loc}$ for every $X\in \mathfrak{g}_1$;
		\item[\rmiii] If $f \in L^2_\loc$ is such that $
		\mathcal{L}f \in S^{2}_{-1,\loc}$, then $f \in S^{2}_{1,\loc}$. 
	\end{itemize}
\end{lemma}

We shall also need the following proposition. It is the analogue for the sub-Laplacian on $G$ of a well known inequality, due to Kato~\cite{Kato2}, involving the Laplacian on $\R^n$. We recall that, if $T_1,T_2$ are two distributions, $T_1\leq T_2$ means that $T_2-T_1$ is a positive (Radon) measure. The $\sign$ function on $\R^d$, $d\geq 1$, or $\C$ is defined as $\frac{z}{|z|}$ when $z\neq 0$, and $0$ when $z=0$.
\begin{proposition}
	\label{KatoSub}
	Let $f\in L^1_{\loc}$ such that $\mathcal{L} f\in L^1_{\loc}$. Then $\mathcal{L}\abs{f}\leq \Re\left[\overline{\sign(f)}\,\mathcal{L}f\right]$.
\end{proposition}	
Its proof is a straightforward adaptation to our setting of~\cite[Theorem X.27]{ReedSimonII}, and is omitted. A similar statement, in a much more general context but with  stronger hypotheses, can be found in~\cite{Simon2}.

\section{Essential self-adjointness of weighted sub-Laplacians}\label{sec:weighted:schr}
In this section, $G$ will be a stratified \emph{non commutative} group, so that $\mathfrak{g}$ has at least two layers and $Q\geq 4$. Henceforth, we adopt the well known notation of writing $(\, \cdot\,,\, \cdot\, )$ for the inner product of elements in some Hilbert space ($L^2$ or $L^2(w)$, see Definition~\ref{defspazipesati}), while $\langle\, \cdot\, ,\,\cdot\, \rangle$ will stand for (sesquilinear) distributional pairings.

\begin{definition}\label{admissibleweight}
	We shall call \emph{admissible weight} a strictly positive and weakly differentiable function $w$ on $G$, such that $w$, $ w^{-1} \in L^\infty_\loc$,  and $\nabla_\Hc w \in L^2_\loc $.
\end{definition}

\begin{definition}\label{defspazipesati}
	Let $w$ be an admissible weight. We denote by $L^2(w)$ the Hilbert space
	\[\{ f: G \to \C \colon f\, \mbox{measurable}, \; \int_{G} |f(y)|^2w(y)\,dy <\infty\}\]
	with the usual scalar product and norm. The space $L^2(w;\mathfrak{g}_1)$ of square-integrable $\mathfrak{g}_1$-valued functions is defined analogously. Moreover, we denote by \[S^2_{1}(w)= \{ f\in L^2(w) \colon \; \nabla_\Hc f \in L^2(w;\mathfrak{g}_1)\}\] the weighted Sobolev space, endowed with the standard norm.
\end{definition}

By routine arguments, one can prove that $S^2_{1}(w)$ is a Hilbert space and that $C^\infty_c$ is dense in it. Therefore, the sesquilinear form
\[
\mathfrak{t}(f,g)\coloneqq ( \nabla_\Hc f, \nabla_\Hc g)_{L^2(w)}, \quad \Dom(\mathfrak{t}) = S^2_{1}(w)
\]
is symmetric, densely defined, accretive, continuous and closed. By the general theory of sesquilinear forms (see, e.g., \cite[IV, Theorem 2.6]{Kato1}) there exists a positive self-adjoint operator $\mathcal{L}^w$ such that
\begin{gather*}
\Dom(\mathcal{L}^w) = \{ f\in S^2_{1}(w)\colon\, \exists \tilde f \in L^2(w)\; \forall \, g\in S^2_{1}(w)\quad (\tilde f,g)_{L^2(w)} = \mathfrak{t}(f,g) \}, \\ \mathcal{L}^w f\coloneqq \tilde f \quad \forall f\in \Dom(\mathcal{L}^w).
\end{gather*}
An easy computation shows that for every $\phi$, $\psi \in C_c^\infty$
\begin{equation*}\label{weightedSL}
(\mathcal{L}^{w} \phi, \psi)_{L^2(w)}= \mathfrak{t}(\phi,\psi) = \left( \mathcal{L}\phi - \frac{\nabla_\mathcal{H} {w}}{{w}} \cdot \nabla_{\Hc}\phi,  \psi\right)_{L^2(w)}.
\end{equation*}
It is then natural to define the operator
\begin{equation}\label{Lw0}
\mathcal{L}^w_0\coloneqq \mathcal{L} - \frac{\nabla_\mathcal{H} {w}}{{w}} \cdot \nabla_{\Hc}, \quad \Dom(\mathcal{L}^w_0) \coloneqq C_c^\infty.
\end{equation}
Notice that, since  $\nabla_\Hc w \in L^2_\loc$ and  $w^{-1}\in L^\infty_\loc$ by Definition~\ref{admissibleweight}, the operator $\mathcal{L}^w_0 $ maps $C^\infty_c$ to $ L^2(w)$. Since for every $f\in S^{2}_{1,\loc}$, we have $\frac{\nabla_\mathcal{H} {w}}{{w}} \cdot \nabla_{\Hc} f \in L^1_\loc$, we may define also the operator
\[
\mathcal{L}^w_1 \coloneqq \mathcal{L} - \frac{\nabla_\mathcal{H} {w}}{{w}} \cdot \nabla_{\Hc} \colon S^{2}_{1,\loc} \to \mathcal{D}',
\]
where the derivatives are to be understood in the distributional sense.  The following theorem, whose proof will occupy the remainder of this section, describes the relations among the operators $\mathcal{L}^w_0$, $\mathcal{L}^w$ and $\mathcal{L}^w_1$.

\begin{theorem}\label{teo}
	If $w$ is an admissible weight, then $\mathcal{L}^w$ is a positive self-adjoint extension of $\mathcal{L}^w_0$ and $\Dom(\mathcal{L}^w) = \{f\in S^2_{1}(w) \colon\, \mathcal{L}^w_1 f\in L^2(w)\}$.
	
	Assume, in addition, that one of the following conditions holds:
	\begin{enumerate}
		\item[$\mathrm{(a)}$] $\nabla_\Hc w \in L^\infty_\loc$ and $\mathcal{L} w\in L^Q_\loc$;
		\item[$\mathrm{(b)}$]  $\frac{\nabla_\Hc w}{(1+N)w}\in L^\infty$ for some (hence any) homogeneous norm $N$ on $G$.
	\end{enumerate}
	Then $\mathcal{L}^w_0$ is essentially self-adjoint, i.e.\ $\mathcal{L}^w$ is the unique self-adjoint extension of $\mathcal{L}^w_0$.
\end{theorem}
\begin{remark}\label{rem1}
	The essential self-adjointness of $\mathcal{L}_0^w$ is well known 
	when the weight $w$ is smooth (see e.g.\cite[Proposition 3.2.1]{Bakryetal}). The classical argument makes use of 
	two ingredients:
	\begin{itemize}
		\item[\rmi] the hypoellipticity of $\mathcal{L}_0^w$:  any $f\in 
		\Dom((\mathcal{L}^w_0)^*)$ such that $(\mathcal{L}^w_0)^*f+f=0$ is 
		smooth;
		\item[\rmii] completeness: there exists an increasing sequence $
		(\psi_k)$ of positive functions in $C^\infty_c$ that converges to 
		$1$ almost everywhere and $|\nabla_\mathcal{H} \psi_k|\le 1/k$ for 
		all $k\ge 1$. We shall refer to the sequence $(\psi_k)$ as an 
		\textit{approximate unity} for $G$.
	\end{itemize}
	Completeness does not depend on the weight and therefore holds 
	also under our assumptions. It suffices to define $\psi_k=\psi\circ
	\delta_{1/k}$, $k\ge 1$, where $\psi $ is a  function  in $C^\infty_c$ 
	such that $0\le \psi \le 1$ and $\psi(x)=1$ if $N(x)\le1$ and $
	\psi(x)=0$ if $N(x)\ge 2$.\par
	Instead, hypoellipticity of $\mathcal{L}_0^w$ fails when the weight 
	$w$ is nonsmooth. However, we shall see that a weaker form of 
	hypoellipticity holds if $\nabla_\mathcal{H} w\in L^\infty_\loc$ and $
	\mathcal{L}w\in L^Q_\loc$ (see Lemma \ref{Bootstrap1}).
\end{remark}
To prove the theorem we shall need two lemmas. In their proofs and henceforth, we shall apply the common rules of derivation to products and compositions of distributions with Lipschitz functions. Each equality may be proved either by density arguments, or by convolution on the left with a suitable approximate identity.

\begin{lemma}\label{lem:6} 
	Let $w$ be an admissible weight.  If 
	\begin{itemize}
		\item[\rmi] either $f,g\in S^2_1(w)$ and  $\mathcal{L}^w_1 f\in L^2(w)$
		\item[\rmii] or $f,g\in S^2_{1,\loc}$, $\mathcal{L}^w_1 f\in L^2_\loc$, $g$ has compact support and $\nabla_\mathcal{H} w\in L^\infty_\loc$,
	\end{itemize} then
	\[
	( \mathcal{L}^w_1 f ,  g)_{L^2(w)}=\langle \nabla_\Hc f, w \nabla_\Hc g\rangle.
	\]
\end{lemma}

\begin{proof}
	We first assume \rmi. Since both sides of the asserted equality are continuous functions of $g\in S^2_{1}(w)$, by density it will be enough to prove our assertion when $g \in C^\infty_c$. 
	If $\mathcal{L}^w_1 f\in L^2(w)$, since $\nabla_\Hc f$  and $\frac{\nabla_\Hc w}{w} \in L^2_\loc$, we infer that $\mathcal{L}f$ and $\frac{\nabla_\Hc w}{w}\cdot \nabla_\Hc f \in L^1_\loc$. Therefore, $\mathcal{L}^w_1 f=\mathcal{L}f- \frac{\nabla_\Hc w}{w}\cdot \nabla_\Hc f  \in L^1_\loc$ and the equality to prove  is implied by:
	\[
	\langle \mathcal{L} f, w g \rangle=\langle \nabla_\Hc f, \nabla_\Hc(w g)\rangle,
	\]
	where the first pairing refers to the duality between $L^1_\loc$ 
	and the space of $L^\infty$ functions with compact support, while 
	the second pairing refers to the duality between $L^2_\loc$ and the 
	space of $L^2$ functions with compact support. Since $wg\in 
	S^2_1$ and has compact support, an integration by parts shows 
	that this identity holds for $f\in C^\infty$. The general case can be 
	reduced to this, by convolving $f$ on the left with an 
	approximate identity $(\eta_k)$ of smooth functions with 
	compact support. Since $\mathcal{L}$ and $\nabla_\mathcal{H}$ 
	commute with left convolution, $\mathcal{L}(\eta_k* f)=\eta_k*
	\mathcal{L} f$ converges to $\mathcal{L} f$ in $ L^1_\loc$ and $
	\nabla_\mathcal{H}(\eta_k* f)=\eta_k*\nabla_\mathcal{H} f$ 
	converges to $\nabla_\mathcal{H} f$ in $L^2_\loc$. The conclusion 
	follows by the continuity of the pairings. This proves case \rmi.
	\par
	Case \rmii\ may be reduced to case \rmi, by multiplying $f$ by a function $\phi\in C^\infty_c$ such that $\phi=1$ in a neighbourhood of the support of $g$ and observing that $f\phi, g\in S^2_{1}(w)$, $\mathcal{L}^w_1(f\phi)=\phi\mathcal{L}^w_1f+f\mathcal{L}^w_1\phi-2\nabla_\mathcal{H} \phi\cdot\nabla_\mathcal{H} f\in L^2(w)$, since $\nabla_\mathcal{H} w\in L^\infty_\loc$, and
	\[
	(\mathcal{L}_1^w f, g)_{L^2(w)}=(\mathcal{L}_1^w (f\phi), g)_{L^2(w)}=(\nabla_\Hc (f\phi), \nabla_\Hc g)_{L^2(w)}=(\nabla_\Hc f, \nabla_\Hc g)_{L^2(w)}.
	\] \end{proof}

\begin{lemma}\label{Bootstrap1}
	Let $w$ be an admissible weight and suppose that $f\in \Dom((\mathcal{L}^w_0)^*)$ satisfies $(\mathcal{L}^w_0)^* f+f=0$. Then
	\begin{itemize}
		\item[\rmi] if $\nabla_\Hc w\in L^\infty_\loc$ then $f\in S^{2}_{1,\loc}$;
		\item[\rmii] if, furthermore, $\mathcal{L} w\in L^Q_\loc$ then $f\in S^{2}_{2,\loc}$.
	\end{itemize}
\end{lemma}

\begin{proof}
	Since $\Dom((\mathcal{L}^w_0)^*) \subseteq L^2(w) \subseteq L^2_{\loc}$, we have that $f\in L^2_\loc$; hence, both $fw$ and $f\nabla_\Hc w$ belong to $ L^2_\loc$. Then, for every $\phi\in C^\infty_c$,
	\begin{equation*}
	\begin{split}
	\langle w(\mathcal{L}_0^w)^*f,\phi \rangle&=	( (\mathcal{L}_0^w)^*f,\phi )_{L^2(w)} \\&= ( f, \mathcal{L}_0^w \phi)_{L^2(w)}\\&= \left( f, \mathcal{L}\phi - \frac{\nabla_\mathcal{H} w}{w} \cdot \nabla_\mathcal{H} \phi \right)_{L^2(w)} \\&= \langle \mathcal{L}(f w),  \phi\rangle + \langle \divv_\Hc(f \nabla_\mathcal{H} w),\phi\rangle.
	\end{split}
	\end{equation*}
	By the arbitrariness of $\phi$, this implies that
	$$
	w(\mathcal{L}^w_0)^* f= \mathcal{L}(f w)+ \divv_\Hc(f \nabla_\Hc w).
	$$By Lemma~\ref{lemmainiz}, (ii), $\divv_\Hc(f\nabla_\mathcal{H} w) \in S^{2}_{-1,\loc}$ since $f\nabla_\mathcal{H} w\in L^2_\loc$. Therefore, 
	\begin{equation}\label{eqdistr2}
	\begin{split}
	\mathcal{L}(f w)&= w(\mathcal{L}^w_0)^* f-\divv_\Hc (f \nabla_\Hc w)\\&= -w f-\divv_\Hc (f \nabla_\Hc w)\in S^{2}_{-1,\loc}.
	\end{split}
	\end{equation}
	Hence, by Lemma~\ref{lemmainiz}, (iii), $f w \in S^{2}_{1,\loc}$. Thus
	\[\nabla_\Hc f = \nabla_\Hc \left(fw w^{-1}\right) = w^{-1}\nabla_\Hc (fw) - f \frac{\nabla_\Hc w}{w} \in L^2_\loc\]
	since both $w^{-1}$ and $\nabla_\Hc w $ belong to $L^\infty_\loc$. Hence $f\in S^{2}_{1,\loc}$.
	\par
	Next, assume that $\mathcal{L}w\in L^Q_\loc$. Since $f\in S^{2}_{1,\loc}$, by Lemma~\ref{lemmainiz}, (i) it is enough to prove that $\mathcal{L}f \in L^2_\loc$. First we prove that $\mathcal{L} (fw) \in L^2_\loc$.
	Recall that $S^2_1$ embeds continuously in $L^{2 Q/(Q-2)}$ (cf.~\cite[Theorem 4.17]{Folland}). Thus $f\in L^{2 Q/(Q-2)}_{\loc}$. By applying H\"older's inequality to the conjugate exponents $Q/2$ and $Q/(Q-2)$, we obtain that
	\begin{equation}\label{fLw}
	f\mathcal{L}w \in L^2_\loc.
	\end{equation}
	Thus, since $\nabla_\Hc f\in L^2_\loc$ and $\nabla_\Hc f \cdot \nabla_\Hc w \in L^2_\loc$, we get  that
	\begin{equation}\label{eqdivv}
	\divv_\Hc(f\nabla_\mathcal{H} w) = \nabla_\mathcal{H} f \cdot \nabla_\mathcal{H} w - f \mathcal{L}w\in L^2_\loc.
	\end{equation}
	Thus, by~\eqref{eqdistr2} 
	\[\mathcal{L}(f w) = -(\divv_\Hc(f\nabla_\mathcal{H} w) + w f) \in L^2_{\loc}. \]
	Since $\nabla_\Hc (f w)\in L^2_\loc$, by Lemma~\ref{lemmainiz}, (i) we get that $f w \in S^{2}_{2,\loc}$.
	Therefore, we can write
	\begin{equation*}
	\begin{split}
	\mathcal{L}f &= \mathcal{L}\left(f ww^{-1}\right) = \mathcal{L}(f w) w^{-1} + fw \mathcal{L}\left(w^{-1}\right) + 2\nabla_\mathcal{H}(fw)\cdot \nabla_\mathcal{H}\left(w^{-1}\right).
	\end{split}
	\end{equation*}
	The first and the third term are in $L^2_\loc$ for what we showed above. As for the second term, since $\mathcal{L}w\in L^Q_\loc$,
	\[
	\begin{split}
	f w \mathcal{L}\left(w^{-1}\right)& = f w \frac{\mathcal{L}w}{w^2} - 2f w \frac{\abs{\nabla_\mathcal{H} w}^2}{w^3}\\& =f\frac{\mathcal{L}w}{w}- 2 f\frac{\abs{\nabla_\mathcal{H} w}^2}{w^2} 
	\end{split}
	\]
	Thus also the third term is in $L^2_\loc$, since $f\mathcal{L}w\in L^2_\loc$ by \eqref{fLw} and $w$, $\nabla_\mathcal{H}w$ are in $L^\infty_\loc$. This concludes the proof.
\end{proof}

\begin{proof}[Proof of Theorem \ref{teo}]
	We begin by proving that
	\begin{equation}\label{caratt}
	\Dom(\mathcal{L}^w) =\left\{ f\in \Dom(\mathfrak{t}) \colon \, \mathcal{L}^{w}_1 f\in L^2(w)\right\}, \quad \mathcal{L}^w f = \mathcal{L}^w_1 f \quad \forall f\in \Dom(\mathcal{L}^w).
	\end{equation}
	The inclusion $\supseteq$ follows from Lemma~\ref{lem:6}. Assume then $f\in\Dom(\mathcal{L}^w)$, and take $\phi\in C_c^\infty$. Then $w^{-1}\phi\in \Dom(\mathfrak{t})$, so that
	\[
	\begin{split}
	\langle \mathcal{L}^w f,\phi\rangle&=\left( \mathcal{L}^w f,
	w^{-1}\phi\right)_{L^2(w)}\\& =\mathfrak{t}\left(f,w^{-1}\phi\right) \\ &=\langle \nabla_\Hc f, \nabla_\Hc\phi-w^{-1}\phi\nabla_\Hc w\rangle\\
	&= \left\langle \mathcal{L} f, \phi\right\rangle-\left\langle \frac{\nabla_\Hc w}{w}\cdot \nabla_\Hc  f, \phi\right\rangle,
	\end{split}
	\]
	and hence $\mathcal{L}^w f=\mathcal{L}^w_1 f$. Then~\eqref{caratt} holds, and it is now easy to see that $C_c^\infty\subseteq \Dom(\mathcal{L}^w)$; thus, $\mathcal{L}^w$ is a self-adjoint extension of $\mathcal{L}^w_0$. 
	
	\smallskip
	
	\textbf{Case (a).} Assume that $\nabla_\Hc w\in L^\infty_\loc$ and $\mathcal{L}w\in L^Q_\loc$. By~\cite[Theorem X.1]{ReedSimonII} it suffices to prove that if $f\in \Dom ((\mathcal{L}^w_0)^*)$ satisfies $(\mathcal{L}^w_0)^* f+f=0$, then $f=0$.  By Lemma \ref{Bootstrap1} $f\in S^{2}_{2,\loc}$. Thus
	\[\nabla_\Hc(f w)=w\nabla_\Hc f+f \nabla_\Hc w,\qquad \divv_\Hc(w\nabla_\Hc f)= -w\mathcal{L} f+\nabla_\Hc w  \cdot \nabla_\Hc f\]
	are in $L^2_\loc$. Therefore, for every $\phi\in C^\infty_c$,
	\[
	\begin{split}
	((\mathcal{L}_0^w)^* f, \phi)_{L^2(w)}&=(f, \mathcal{L}_0^w \phi)_{L^2(w)}\\
	&=\langle f w, \mathcal{L}\phi\rangle-\langle f\nabla_\Hc w, \nabla_\Hc \phi\rangle\\
	&=\langle \nabla_\Hc (f w), \nabla_\Hc \phi\rangle-\langle f \nabla_\Hc w, \nabla_\Hc \phi\rangle\\
	&=\langle w\nabla_\Hc f, \nabla_\Hc \phi\rangle\\
	&=-\langle \divv_\Hc (w \nabla_\Hc f), \phi\rangle\\
	&= \langle w\mathcal{L} f-\nabla_\Hc w \nabla_\Hc f, \phi\rangle\\
	&=\langle \mathcal{L}_1^w f, w \phi\rangle\\
	&= (\mathcal{L}_1^w f, \phi)_{L^2(w)}.
	\end{split}
	\]
	Hence, \[\mathcal{L}_1^w f=(\mathcal{L}_0^w)^* f\in L^2(w).\]
	We are now able to argue as Strichartz~\cite{Strichartz}. Take an approximate unity $(\psi_k)$ for $G$ as in Remark~\ref{rem1}. Hence $\psi_k^2 f \in S^2_{1}(w)$ and
	\[
	\begin{split}
	0\geq -( \psi_k^2 f,f)_{L^2(w)} &= ( \psi_k^2 f, (\mathcal{L}_0^w)^*f )_{L^2(w)}\\
	& = ( \psi_k^2 f, \mathcal{L}_1^w f)_{L^2(w)}\\&= ( \nabla_\mathcal{H} (\psi_k^2 f), \nabla_\mathcal{H} f )_{L^2(w)}\\
	&= ( 2\psi_k (\nabla_\mathcal{H} \psi_k) f, \nabla_\mathcal{H} f )_{L^2(w)} + ( \psi_k^2 \nabla_\mathcal{H} f, \nabla_\mathcal{H} f)_{L^2(w)}. 
	\end{split}
	\]
	where the third equality holds by Lemma \ref{lem:6} \rmi. Therefore,
	\[
	\begin{split}
	\norm{\psi_k \nabla_\mathcal{H} f}^2_{L^2(w)} &=\langle \psi_k^2 \nabla_\mathcal{H} f, \nabla_\mathcal{H} f\rangle_{L^2(w)}\\
	&\leq-\langle 2\psi_k (\nabla_\mathcal{H} \psi_k) f, \nabla_\mathcal{H} f \rangle_{L^2(w)}\leq 2\norm{\psi_k \nabla_\mathcal{H} f}_{L^2(w)} \norm{f \nabla_\mathcal{H} \psi_k }_{L^2(w)} ,
	\end{split}
	\]
	so that
	\[
	\begin{split}
	\norm{\psi_k \nabla_\mathcal{H} f}_{L^2(w)}\leq 2\norm{\nabla_\mathcal{H} \psi_k}_{L^\infty(w)} \norm{f}_{L^2(w)}.
	\end{split}
	\]
	By Fatou's lemma, finally,
	\[
	\begin{split}
	\norm{\nabla_\mathcal{H} f}_{L^2(w)}&\leq\liminf_{k\to\infty} \norm{\psi_k\nabla_\mathcal{H} f}_{L^2(w)}
	\leq\lim_{k\to\infty} \frac{2}{k}\norm{\nabla_\mathcal{H} \psi_1}_{L^\infty(w)} \norm{f}_{L^2(w)}=0,
	\end{split}
	\]
	whence $\nabla_\mathcal{H} f =0$. Therefore, $f=0$ and Case \textbf{(a)} is proved.

	\smallskip	
	
	\textbf{Case (b).} Assume that $\frac{\nabla_\Hc w}{(1+N)w}\in L^\infty$.  To prove that $\mathcal{L}^w_0$ is essentially self-adjoint it suffices to show that $C_c^\infty$ is dense in $\Dom(\mathcal{L}^w)$. First we show that the space of functions in $\Dom(\mathcal{L}^w)$ with compact support is dense in $\Dom(\mathcal{L}^w)$.
	
	Let $f\in\Dom(\mathcal{L}^w)$ and take an approximate unity $(\psi_k)$ as in Remark~\eqref{rem1}. Then $\psi_kf\in\Dom(\mathcal{L}^w)$ for all $k$ by \eqref{caratt}, since
	\begin{equation}\label{eqdramm}
	\mathcal{L}^w_1 (\psi_k f)=(\mathcal{L}^w_0 \psi_k) f+\psi_k(\mathcal{L}^w_1 f)-2\nabla_\Hc \psi_k\cdot\nabla_\Hc f,
	\end{equation}	
	is in $L^2(w)$. Moreover $\mathcal{L}^w_1(\psi_k f)$ tends to $\mathcal{L}^w_1 f$ 
	in $L^2$. Indeed, 
	$
	\lim_k \psi_k(\mathcal{L}^w_1 f)=
	\mathcal{L}^w_1 f
	$ 
	in $L^2(w)$ by dominated convergence and $
	\lim_k \norm{\nabla_\Hc \psi_k\cdot\nabla_\Hc f}_{L^2(w)}=0$ since $
	\norm{\nabla_\mathcal{H}\psi_k}_{\infty}\le C/k$.   Thus it remains only to prove that 
	$$
	(\mathcal{L}^w_0 \psi_k) f=(\mathcal{L}\psi_k)f-w^{-1}{\naH w}\cdot
	\naH \psi_k \ f
	$$
	tends to zero in $L^2(w)$. Indeed, on the one hand, $\norm{\mc{L}\psi_k}_{\infty}\le Ck^{-2}$. On the other hand, since  the support 
	of $\naH{\psi_k}$ is contained in $A_k=\{x: k\le N(x)\le 2k\}$, we 
	have that $k\sim (1+N)$ on the support of $\naH{\psi_k}$. Thus $\norm{w^{-1}\naH{w}\cdot\naH{\psi_k}\ f}_{L^2(w)}\le C\norm{f}
	_{L^2(A_k,w)}$, which tends to zero.
	
	It remains only to prove that if $f \in \Dom(\mathcal{L}^w)$ has 
	compact support then there exists a sequence $(f_k)$ of functions in  
	$C^\infty_c$ such that $\lim_k f_k=f$ and $\lim_k \mc{L}^w_0f_k=
	\mc{L}^wf$ in $L^2(w)$. To this end, define $f_k=\eta_k*f$, where $(\eta_k)$ is an approximate identity for the convolution in $C^\infty_c$ and observe that
	\[
	\begin{split}
	\mathcal{L}^w_0(\eta_k*f)&=\mathcal{L} (\eta_k*f)-\frac{\nabla_\Hc w}{w}\cdot \nabla_\Hc(\eta_k*f) 
	\\& =\eta_k*\mathcal{L} f-\frac{\nabla_\Hc w}{w} \cdot \eta_k* (\nabla_\Hc f).
	\end{split}
	\]
	The conclusion follows, since for functions supported in a fixed compact set convergence in $L^2$ is equivalent to convergence in $L^2(w)$. This completes the proof.
\end{proof}

The following corollary specializes the result of Theorem~\ref{teo} to weights of the form  $w_\alpha= e^{-N^\alpha}$, where $N$ is any homogeneous norm on  $G$ and  $\alpha>0$. 

\begin{corollary}\label{essselfadj:walpha}
	$\mathcal{L}^{w_\alpha}$ is a self-adjoint extension of  $\mathcal{L}^{w_\alpha}_0$ for every $\alpha>0$. When $\alpha \geq 1$, $\mathcal{L}^{w_\alpha}$ is the unique self-adjoint extension of  $\mathcal{L}^{w_\alpha}_0$.
\end{corollary}	
\begin{proof} Outside the origin 
	\begin{equation}\label{nablawalpha}
	\nabla_\mathcal{H} w_\alpha = - \alpha w_\alpha  N^{\alpha -1} \nabla_\mathcal{H} N,
	\end{equation}
	and
	\begin{equation}\label{Lwalpha}
	\mathcal{L}w_\alpha= w_\alpha\left[-\alpha^2 N^{2\alpha-2}\abs{\nabla_\Hc N}^2 +\alpha(\alpha-1)N^{\alpha-2}\abs{\nabla_\Hc N}^2 + \alpha N^{\alpha-1} \mathcal{L}N\right].
	\end{equation}
	Since the right hand sides of both identities are in $L^1_\loc$, it is not hard to see that they coincide with $\nabla_\Hc w_\alpha$ and $\mathcal{L}w_\alpha$, respectively, in the sense of distributions on $G$. Thus the weight $w_\alpha$ is admissible for all $\alpha>0$. Moreover,
	\begin{itemize}
		\item[(a)] if $\alpha>1$ then $\naH{w_\alpha}\in L^\infty_\loc$ and $\mathcal{L}w_\alpha\in L^Q_\loc$;
		\item[(b)] $\naH{w_1}\in L^\infty_\loc$ and $w_1^{-1} \naH{w_1}\in L^\infty$.
	\end{itemize}
	The conclusion follows by Theorem~\ref{teo}.
\end{proof}

\section{Schr\"odinger operators}\label{SchrodingerSection}
Let $w$ be an admissible weight such that $\mathcal{L}w \in L^1_\loc$. Instead of dealing directly with the weighted operator $\mathcal{L}^{w}$, it is often more convenient to consider a Schr\"odinger operator unitarily equivalent to it. Indeed,  if $U\colon 
L^2(w) \rightarrow L^2$ is the unitary map defined by $U f\coloneqq 
f \sqrt{w}$, then $\mathcal{L}^w$ is unitarily equivalent to 
$U\mathcal{L}^w U^{-1}$ with domain $U\Dom(\mathcal{L}^w)$. If 
the weight $w$ is smooth then $U$ maps $C^\infty_c$ onto itself 
and a straightforward computation gives
\begin{equation}\label{conj}
U\mathcal{L}^w U^{-1} \phi =(\mathcal{L}+V)\phi \qquad\forall \phi\in C^\infty_c,
\end{equation}
where  
\begin{equation}\label{potential}
V= -\frac{1}{4}\frac{\abs*{\nabla_\mathcal{H}w}^2}{w^2} - \frac{1}{2}\frac{\mathcal{L}w}{w}.
\end{equation}
If the operator $\mathcal{L}+V$ is essentially self-adjoint on $C^\infty_c$ then the identity \eqref{conj} completely determines $\mathcal{L}^w$ and the study of the spectrum of $\mathcal{L}^w$ can be reduced to that of the closure of $\mathcal{L}+V$. If the weight $w$ is not smooth matters become more delicate because, on the one hand, the derivatives in \eqref{potential} must be interpreted in the sense of distributions and, on the other hand,  one must be careful in identifying the domains of the operators. 
The following lemma shows that the identity~\eqref{conj} still holds in the distributional sense if we assume sufficient regularity of the weight $w$.
\begin{lemma}\label{coniugio}
	If $\nabla_\Hc w\in L^\infty_\loc$ and $\mathcal{L}w \in L^2_\loc$, then $V\in L^2_\loc$ and 
	$$
	U^{-1}\phi\in\Dom(\mathcal{L}^w)
	,\qquad U\mathcal{L}^w U^{-1} \phi=\mathcal{L}\phi+V\phi \qquad\forall \phi\in C^\infty_c.
	$$ 
\end{lemma}
\begin{proof}
	Let $\phi$ be a function in $C^\infty_c$. A straightforward computation shows that $U^{-1}\phi\in S^2_{1}(w)$ and $\mathcal{L}^w_1 U^{-1}\phi\in L^2(w)$. Thus $U^{-1}\phi \in \Dom(\mathcal{L}^w)$, by Theorem \ref{teo}.  The identity $U\mathcal{L}^w U^{-1} \phi=\mathcal{L}\phi+V\phi$  follows easily.
\end{proof}
The following theorem is an analogue of a classical result for the Laplacian on $\R^n$ (see Kato~\cite{Kato2}). It will be used in Proposition~\ref{uneq} to prove that, when the weight $w$ is sufficiently regular, the operator $\mathcal{L}^w$ is unitarily equivalent to the closure of $\mathcal{L}+V$ on $C^\infty_c$. 
\begin{theorem}\label{thmselfadj}
	Let $V\in L^2_\loc$, and suppose that $V$ is essentially bounded from below. Then $\mathcal{L}+V$, with domain $C_c^\infty$, is essentially self-adjoint on $L^2$. 
\end{theorem}
\begin{proof}
	Up to adding a constant, we can assume that $V\geq 0$. By~\cite[Theorem X.1]{ReedSimonII} it is enough to prove that, if $(\mathcal{L}+V)^*f + f=0$ for some $f\in \Dom((\mathcal{L}+V)^*)$, then $f=0$. This holds in particular if we show that, if $(\mathcal{L}+V)f + f=0$ in the sense of distributions for some $f\in L^2$, then $f=0$.
	
	Suppose, then,
	\begin{equation*}
	\mathcal{L}f + Vf +f=0.
	\end{equation*}
	Since $f\in L^2$ and $V\in L_{\loc}^2$, $V f\in L_{\loc}^1$.
	Moreover, $f\in L_{\loc}^1$ and hence $\mathcal{L}f=-(Vf + f)\in L_{\loc}^1$. Therefore, by Proposition~\ref{KatoSub}
	\begin{equation*}
	\begin{split}
	\mathcal{L}\abs*{f} \leq \Re\left[\overline{\sign (f)} \mathcal{L}f\right]=- \Re\left[ \overline{\sign( f)}(V+1)f\right]=-(V+1)\abs*{f},
	\end{split}
	\end{equation*}
	so that
	\begin{equation*}
	(\mathcal{L}+I)\abs*{f}\leq -V\abs*{f}\leq 0.
	\end{equation*}
	In particular, for every nonnegative $\phi \in C_c^\infty$,
	\begin{equation}
	\langle \abs*{f}, (\mathcal{L}+I)\phi \rangle= \langle (\mathcal{L}+I)|f|, \phi\rangle \leq 0.
	\label{eqparzter}
	\end{equation}
	More generally, if $g$ is a nonnegative function in $\Dom(\mathcal{L})$, by approximating $g$ with a sequence $(\phi_k)$ of nonnegative functions in $C_c^\infty$  which tends to $g$ in $\Dom(\mathcal{L})$\footnote{It suffices to choose $\phi_k= \psi_k(\eta_k* g)$, where $(\psi_k)$ is a nonnegative approximate unity and $(\eta_k)$ is a nonnegative approximate identity.}, we obtain that
	$$
	\langle \abs*{f}, (\mathcal{L}+I)g\rangle=\lim_{k\to\infty}\langle \abs*{f}, (\mathcal{L}+I)\phi_k\rangle \le 0.
	$$
In particular, if
	\begin{equation*}
	g=(\mathcal{L}+I)^{-1}\psi = \int_0^\infty e^{-s}e^{-s\mathcal{L}}\psi \,\dd s,
	\end{equation*}
	where $\psi$ is a nonnegative function in $C^\infty_c$, then $g\in\Dom(\mathcal{L})$ and is nonnegative, since the semigroup $e^{-s\mathcal{L}}$ is positivity preserving (cf.~\cite[Theorem 2.7]{Ouhabaz}).
	Thus
	\begin{equation*}
	\langle \abs*{f},\psi\rangle \leq 0 \quad \forall \psi\in C_c^\infty,\; \psi\geq 0.\
	\end{equation*}
	It follows that $f=0$, and this concludes the proof.
\end{proof}
\begin{notation}
When $T$ is an operator with domain $\Dom(T)$ and $\mathscr{D}\subseteq \Dom(T)$, we  denote by $(T,\mathscr{D})$ the restriction of $T$  to $\mathscr{D}$. 
\end{notation}

\begin{proposition}\label{uneq}
	Let $w$ be an admissible weight such that $\nabla_\Hc w \in L^\infty_\loc$ and $\mathcal{L}w \in L^2_\loc$. Suppose that $V$ in \eqref{potential} is bounded from below. Then $\mathcal{L}^w$ is unitarily equivalent to the closure of $(\mathcal{L}+V, C_c^\infty)$.
\end{proposition}	
\begin{proof} If $\phi$ is a function in $C^\infty_c$ then $U^{-1}\phi \in \Dom(\mathcal{L}^w)$, by Lemma~\ref{coniugio}. Hence $C^\infty_c\subset U\Dom(\mathcal{L}^w)$ and $(U\mathcal{L}^wU^{-1}, U\Dom(\mathcal{L}^w))$ is a self-adjoint extension of $(\mathcal{L}+V,C^\infty_c)$. Since the latter operator is essentially self-adjoint by Theorem \ref{thmselfadj},  $(U\mathcal{L}^wU^{-1}, U\Dom(\mathcal{L}^w))$ is the closure of $(\mathcal{L}+V,C^\infty_c)$. This proves the proposition.
\end{proof}

\section{Discreteness of the spectrum on Métivier groups}\label{sec:discr}
This section is based on a result of Simon~\cite{Simon1} on the discreteness of the spectrum of the Schr\"odinger operator $-\Delta+V$ on $\R^n$,  when the sublevel sets of the potential $V$ are \textit{polynomially thin}.  Simon's definition of polynomial thinness can be  adapted to the setting of stratified groups as follows. For a subset $E$ of $G$, we write $\abs{E}$ to denote its measure with respect to the measure $\dd y$.
\begin{definition} Given $\ell>0$, a set $\Omega \subseteq G$ is said to be \textit{$\ell$-polynomially thin} if
	\[\int_{\Omega} \abs{\Omega \cap B(y,r)}^{\ell} \,d y <\infty \]
	for every $r>0$. Here $B(y,r)$ is the ball induced by the left-invariant metric associated with any homogeneous norm on $G$.
\end{definition}
\begin{theorem}\label{polythin}
	Let $V$ be a potential bounded from below. Assume that for every $M>0$ there is $\ell>0$ such that $\Omega_{M}\coloneqq \{y\in G \colon \, V(y)\leq M\}$ is $\ell$-polynomially thin. Then there exists a self-adjoint extension of $(\mathcal{L}+V, C^\infty_c )$ with purely discrete spectrum.
\end{theorem}
The proof is a simple adaptation to stratified groups of ~\cite[Theorem 3]{Simon1}. Indeed, Simon's proof relies only on some properties of the heat kernel on $\R^n$ which are true also for the heat kernel on $G$.
This theorem holds on general stratified groups, but henceforth we shall restrict to the particular case where $G$ is a Métivier group.

\begin{definition}(cf.~\cite{Metivier})
Let $\mathfrak{g}$ be a two-step Lie algebra with centre $\mathfrak{z}$, and denote by $\mathfrak{h}$ a complement of $\mathfrak{z}$ in $\mathfrak{g}$. Let $G$ be the connected, simply connected Lie group associated with $\mathfrak{g}$. We say that $G$ is an H-type group in the sense of Métivier, or simply a Métivier group, if for every $\eta \in \mathfrak{z}^*$ the skew-symmetric bilinear form on $\mathfrak{h}$ \[B_\eta (X,Y) \coloneqq \eta([X,Y])\]
is non-degenerate whenever $\eta \neq 0$.
\end{definition}
From now on, $G$ will be a Métivier group. We shall endow $\mathfrak{g}$ with some inner product $(\,\cdot\, , \,\cdot\,)$ and choose $\mathfrak{h}\coloneqq \mathfrak{z}^\perp$, so that $\mathfrak{g} = \mathfrak{h} \oplus \mathfrak{z}$ is an orthogonal stratification~\eqref{stratification} of $\mathfrak{g}$ (cf.~\cite[Section 3.7]{Bonfiglioli}).

It is very convenient to realize $G$ as $\R^{2 n}\times \R^m$, for some $n,m\in \N$, via the exponential map. More precisely, we shall denote by $(x,t)$ the elements of $G$, where  $x\in \R^{2n}$ and $t \in \R^m$. We denote by  $(e_1,\dots,e_{2n})$ and $(u_1,\dots,u_m)$ the standard basis of $\R^{2 n}$ and $\R^{m}$ respectively. Under this identification, the Haar measure $d y$ is the Lebesgue measure. If we define, for $T\in \mathfrak{z}$, a map $J_T \colon \mathfrak{h} \to \mathfrak{h}$ such that 
\[(J_T X, Y)= (T,[X,Y]) \quad \forall \, X,Y\in \mathfrak{h},\]
then the exponential map identifies the maps $\{J_T \colon T\in \mathfrak{z}\}$ with $2n \times 2n$ skew symmetric matrices $\{J_t \colon t\in \R^m\}$ and the group law on $G$ is
\[
	(x,t)\cdot (x',t') = \left(x+x',t+t' + \frac{1}{2} \sum_{k=1}^m (J_{u_k}x,x') u_k\right)
	\]	
	for every $(x,t),(x',t')\in \R^{2 n}\times \R^m$. By definition of Métivier group, the map $J_t$ is non-degenerate whenever $t\neq 0$.

A basis of left-invariant vector fields for $\mathfrak{g}$ is
\[X_j = \partial_{x_j} + \frac{1}{2}\sum_{k=1}^{m}( J_{u_k} x,e_j) \partial_{t_k},\quad j=1,\dots,2n; \qquad T_k = \partial_{t_k}, \quad k=1,\dots,m.\]
In particular, $(X_j)_{1\leq j \leq 2n}$ is a basis for the first layer $\mathfrak{h}\cong \R^{2 n}$.

From now on, we fix the homogeneous norm
\[ N(x,t)=\left(\abs{x}^4+16\, \abs{t}^2\right)^{\sfrac{1}{4}} \]
and the  associated left invariant (quasi) distance $d$. A pseudo-triangle inequality for $N$ holds,
\[N((x,t)\cdot (\xi,\tau))\leq \gamma \left[N(x,t)+N(\xi,\tau)  \right]\]
for some $\gamma>0$ which depends on $N$. In the particular case when $G$ is an H-type group with respect to the inner product $(\,\cdot\,,\,\cdot\,)$, such norm $N$ is the norm appearing in the fundamental solution of the sub-Laplacian~\cite{Kaplan}.

Our main theorem is the following, and its proof will occupy the remainder of the paper.
\begin{theorem}\label{DSalpha}
	If $0<\alpha<1$ no self-adjoint extension of $\mathcal{L}^{w_\alpha}_0$ has purely discrete spectrum. If $1\le \alpha \le 2$, $\mathcal{L}^{w_\alpha}$ does not have purely discrete spectrum. If $\alpha>2$, the spectrum of $\mathcal{L}^{w_\alpha}$ is purely discrete.
\end{theorem}
To prove Theorem~\ref{DSalpha} we reduce matters to studying the spectrum of the Schr\"odinger operator $\mathcal{L} + V_\alpha$
obtained by conjugating $\mathcal{L}^{w_\alpha}$ with the isometry $U_\alpha \colon L^2(w_\alpha)\ni f\mapsto f \sqrt{w_\alpha} \in L^2$, as explained in the previous section.
\begin{lemma}\label{casononcpt}
	Let $\mathcal{L} + V_\alpha$ be the Schr\"odinger operator associated with $\mathcal{L}^{w_\alpha}$. Then, for every $(x,t)\in G^*$,
\begin{equation}\label{potenzialealpha}
	 N^{2 \alpha-4} \abs{x}^2 \left(c_{\alpha,1} - \frac{c_{\alpha,2}}{N(x,t)^\alpha}\right) \leq
	V_\alpha(x,t)\leq N^{2 \alpha-4} \abs{x}^2 \left(c_{\alpha,3} - \frac{c_{\alpha,4}}{N(x,t)^\alpha}\right) 
\end{equation}
for some $c_{\alpha,1}, c_{\alpha,2}, c_{\alpha,3}>0$ and $c_{\alpha,4}\in\R$. If $\alpha\ge 2$, then $V_\alpha$ is bounded from below.
\end{lemma}

\begin{proof}
Observe first that
\begin{equation*}
\begin{split}
	V_\alpha&= -\frac{1}{4}\frac{\abs*{\nabla_\mathcal{H}w_\alpha}^2}{w_\alpha^2} - \frac{1}{2}\frac{\mathcal{L}w_\alpha}{w_\alpha}\\&= \frac{1}{4} \alpha^2 N^{2\alpha-2} \abs{\nabla_\Hc N}^2- \frac{1}{2}\alpha(\alpha-1)N^{\alpha-2}\abs{\nabla_\Hc N}^2 + \frac{1}{2}\alpha N^{\alpha-1} \mathcal{L}N
	\end{split}
\end{equation*}
on $G^*$. It is therefore enough to compute $\abs{\nabla_\Hc N}^2$ and $\mathcal{L} N$. Easy computations show that
\begin{equation*}
\abs*{\nabla_\Hc N}^2(x,t)= 
\frac{\abs{x}^2}{N^6}\left( \abs{x}^4+16 \abs{t}^2 \abs*{J_{\sign(t)} \sign(x) }^2 \right) ,
\end{equation*}
\[ \mathcal{L} N (x,t)= 
 \frac{3}{N}\abs{\nabla_\Hc N}^2-\frac{\abs{x}^2}{N^3}\left(2 + 2 n + 2\sum_{k=1}^m \abs*{J_{u_k} \sign(x)}^2\right). \]
Let $c_0$ and $C_0$ be the minimum and maximum, respectively, of $\abs{J_{t} x}^2$ as $\abs{x}=\abs{t}=1$. Since $J_t$ is non-degenerate for $t\neq 0$,
\begin{equation*}
	\min \{ \abs{J_{t}x} \colon x \in \R^{2n}, \, t\in \R^m, \, \abs{x}=\abs{t}=1 \}>0,
	\end{equation*}
hence $c_0>0$. Let $C\coloneqq \max(C_0, 1)$ and $c\coloneqq \min(c_0,1)$. Thus
 \[
 c \frac{\abs{x}^2}{N^2} \leq \abs{\nabla_\Hc N}^2\leq C\frac{\abs{x}^2}{N^2}
 \]
and
 \[
 \frac{\abs{x}^2}{N^3}(3 c- (2 n+2) - 2 m C_0)\leq \mathcal{L} N\leq \frac{\abs{x}^2}{N^3}(3  C- (2 n+2) -2 m c_0).
 \]
Therefore~\eqref{potenzialealpha} holds with
\[c_{\alpha,1}=\frac{c \alpha^2}{4}, \quad c_{\alpha,2}= \frac{C\alpha^2}{2} -\frac{\alpha}{2}(4 c- 2n-2- 2 m C_0  ),\]
\[ c_{\alpha,3}=\frac{C \alpha^2}{4}\, \quad c_{\alpha,4}= \frac{c\alpha^2}{2} -\frac{\alpha}{2}(4 C- 2n-2- 2 m c_0  ) .  \]
Since $4c -2n-2 \leq 0$, it is easily seen that $c_{\alpha,2}>0$. The boundedness from below of $V_\alpha$ when $\alpha\ge2$  is a straightforward consequence of \eqref{potenzialealpha}.
\end{proof}

\begin{remark}
When $J_t$ is $\abs{t}$-times an isometry for every $t\in \R^m$ and hence $G$ is an H-type group, by~\eqref{potenzialealpha} we get
\[V_\alpha(x,t)=\frac{1}{4} \alpha^2 N(x,t)^{2\alpha -4}\abs{x}^2 - \frac{1}{2} \alpha (Q+\alpha -2) N(x,t)^{\alpha -4} \abs{x}^2\]
as already shown by Inglis~\cite{Inglis}.
\end{remark}

\begin{corollary} For every $\alpha\ge 2$, $(\mathcal{L}+V_\alpha, C_c^\infty)$ is essentially self-adjoint and
	$\mathcal{L}^{w_\alpha}$ is unitarily equivalent to the closure of $(\mathcal{L}+V_\alpha, C_c^\infty)$.
\end{corollary}
\begin{proof}
	If $\alpha\ge2$, then $\nabla_\Hc w_\alpha \in L^\infty_\loc$ and $\mathcal{L}w_\alpha \in L^Q_\loc \subset L^2_\loc$ by \eqref{nablawalpha} and \eqref{Lwalpha}, while $V_\alpha$ is bounded from below by Lemma~\ref{casononcpt}. The conclusion follows from Proposition~\ref{uneq}
\end{proof}	
\begin{remark}\label{alphale2}
	When $0<\alpha< 2$, $V_\alpha$ is not bounded from below,  and we do not know whether $(\mathcal{L}+V_\alpha, C^\infty_c)$ is essentially self-adjoint. Therefore we cannot conclude that $\mathcal{L}^{w_\alpha}$ is unitarily equivalent to the closure of $(\mathcal{L}+V_\alpha, C_c^\infty)$. However, since the weight $w_\alpha$ is smooth in $G^*$,  the operators $(\mathcal{L}^{w_\alpha},C^\infty_c(G^*))$ and  $(\mathcal{L}+V_\alpha, C_c^\infty(G^*))$ are unitarily equivalent. As we shall see this will suffice to show that no self-adjoint extension of $(\mathcal{L}^{w_\alpha},C^\infty_c(G^*))$ has discrete spectrum when $0<\alpha\le2$. 
\end{remark}

\begin{definition}
	Given $\alpha> 0$ and $M>0$, we set \[\Omega_{\alpha,M}\coloneqq\{(x,t)\in G \colon\ V_\alpha(x,t)\leq M\}.\]
\end{definition}
\begin{lemma}\label{lemmathin}
	Let $\alpha>2$. Then, for every $M>0$ and $r>0$ there exists $k=k(\alpha,M,r)$ such that, for every $(x,t) \in\Omega_{\alpha,M}$ with $|t|>k$,
\[\abs{\Omega_{\alpha,M} \cap B((x,t),r)} \leq C(\alpha,M,r) \, |t|^{n\left(2-\alpha\right)}. \]
\end{lemma}

\begin{proof} Since $V_\alpha \to +\infty$ when $\abs{x}\to +\infty$, for every $M$ there exists $c=c(\alpha,M)>0$ such that 
\begin{equation}\label{cylinder}
\abs{x}\leq c(\alpha,M) \qquad\forall (x,t)\in\Omega_{\alpha,M}.
\end{equation}
Let $R=\gamma( r+c(\alpha,M))$.  Then for every $(x,t)\in\Omega_{\alpha,M}$ the ball $B((x,t),r)$ is contained in  $B((0,t),R)$ because the quasi-distance between the centres is $\abs{x}\le c(\alpha,M)$. Thus
\[
\Omega_{\alpha,M} \cap B((x,t),r) \subseteq \Omega_{\alpha,M} \cap B((0,t), R)\qquad\forall (x,t)\in \Omega_{\alpha,M}.
\]
Since  $B((0,t),R)$ is contained in the cylinder $C_R(0,t)=\{(\xi,\tau): \abs{\xi}\le R, \abs{\tau-t}\le R^2\}$, it suffices to estimate the measure of $\Omega_{\alpha,M} \cap C_R(0,t)$. If $k>R^2$, then for all $\abs{t}>k$ and $(\xi,\tau)\in C_R(0,t)$
\[
	N(\xi,\tau) \geq \abs{\tau}^{1/2} \geq (\abs{t}-R^2)^{1/2} > (k-R^2)^{1/2}.
\]
Thus, if we choose $k$ sufficiently large, the quantity $\left(c_{\alpha,1} - \frac{c_{\alpha,2}}{N(\xi,\tau)^{\alpha}}\right)$ is bounded from below by a positive constant  $K=K(\alpha,R,k)$. Hence 
\begin{align*}
\abs*{\Omega_{\alpha,M}\cap C_R(0,t)}&\le \abs*{\left\{ (\xi,\tau)\colon \abs{\tau}^{\alpha - 2}\abs{\xi}^2 \leq \frac{M}{K},\; |\xi|\leq R,\; \abs{t-\tau}\leq R^2 \; \right\}} \\ 
&\le\int_{\abs{\tau-t}\le R^2}\int_{\abs{\xi}^2\le \frac{M}{K}\abs{\tau}^{2-\alpha}} d\xi\,d\tau  \\ 
&\lesssim \int_{B(t,R^2)}  \frac{1}{\abs{\tau}^{n(\alpha -2)}}\, d\tau \lesssim \abs{t}^{n(2-\alpha)}.
\end{align*}
This completes the proof.
\end{proof}

\begin{proposition}\label{nonempty}
	If $0<\alpha\le 2$, then no self-adjoint extension of $(\mathcal{L}+V_\alpha,C_c^\infty(G^*))$ has purely discrete spectrum. If $\alpha>2$, then the unique self-adjoint extension of $(\mathcal{L}+V_\alpha, C_c^\infty)$ has purely discrete spectrum.
\end{proposition} 

\begin{proof}
	Suppose that $0<\alpha\le 2$, and take any self-adjoint extension $T_\alpha$ of $(\mathcal{L}+V_\alpha,C_c^\infty(G^*))$. Fix  any $\lambda$ in the resolvent set of $T_\alpha$. We shall prove that $(\lambda I + T_\alpha)^{-1}$ is non-compact.
	Consider a nonzero $C^\infty$ function $\psi \colon G \rightarrow \C$ supported in $B(e,1)$ and for each $n\in\N$ let $\psi_n$ be the left translate of  $\psi$ by $(0,n)$, i.e.
	\[
	\psi_n(x,t) \coloneqq L_{(0,n)}\psi(x,t) = \psi ((0,-n)(x,t))=\psi (x,t-n),
	\]
	Then $\Supp{\psi_n}\subseteq \mathrm{B}((0,n),1)$; hence $\psi_n\in C^\infty_c(G^*)\subseteq\Dom(T_\alpha)$ for every $n\geq 1$. Denote by $\mathcal{C}$ the set $\{(x,t)\colon \abs*{x}\leq 1, N(x,t)\geq 1\}$; then $C\coloneqq \norm{V_\alpha\chi_\mathcal{C}}_\infty$ is finite by Lemma~\ref{casononcpt}, and $\Supp{\psi_n}\subseteq \mathcal{C}$ for $n\geq 2$.
	Define $f_n \coloneqq (\lambda I + T_\alpha)\psi_n=(\lambda I+ \mathcal{L}+V_\alpha)\psi_n$ for $n\geq 2$. Then,
	\[
	\begin{split}
	\norm{f_n}_2=\norm{(\lambda I + \mathcal{L} + V_\alpha)\psi_n}_2 &\leq \abs{\lambda} \norm{\psi_n}_2 + \norm{\mathcal{L}\psi_n}_2 + \norm{V_\alpha\chi_\mathcal{C}}_\infty \norm{\psi_n}_2\\
	&\leq (\abs{\lambda}+C)\norm{\psi}_2+\norm{\mathcal{L}\psi}_2,
	\end{split}
	\] 
	since 
	\[
	\norm*{\mathcal{L}\psi_n}_2 = \norm*{\mathcal{L}L_{(0,n)}\psi}_2 = \norm*{L_{(0,n)}\mathcal{L}\psi}_2 =  \norm*{\mathcal{L}\psi}_2,
	\]
	by the left-invariance of $\mathcal{L}$. Thus $(f_n)$ is a bounded sequence in $L^2$ such that $(\lambda I + T_\alpha)^{-1} f_n=\psi_n$ does not admit any convergent subsequence, since
	\begin{equation*}
	\norm*{\psi_n - \psi_m}_2^2 = \norm*{\psi_n}_2^2 + \norm*{\psi_m}_2^2 = 2\norm*{\psi}_2^2. 
	\end{equation*}
	Hence the operator $(\lambda I +T_\alpha)^{-1}$ is non-compact, so that $T_\alpha$ cannot have purely discrete spectrum (cf.~\cite[Theorem 11.3.13]{Oliveira}). 
	
	\smallskip
	
	Let now $\alpha>2$. Since $V_\alpha$ is bounded from below, by Theorem \ref{polythin} it is enough to prove that for every $M$ the sublevel set $\Omega_{\alpha,M}$ is $\ell$-polynomially thin for a suitable $\ell >0$. Choose then $r>0$. First, for any $\ell>0$, write
	\[ \int_{\Omega_{\alpha,M}}\abs{\Omega_{\alpha,M} \cap B((x,t),r)}^\ell \,dt\,dx = \left(\int_{\Omega_{\alpha,M}^1} + \int_{\Omega_{\alpha,M}^2}\right)\abs{\Omega_{\alpha,M} \cap B((x,t),r)}^\ell\,dt\,dx\]
	where
	\[\Omega_{\alpha,M}^1 \coloneqq \Omega_{\alpha,M} \cap \{|t|\leq k\}, \qquad \Omega_{\alpha,M}^2 \coloneqq \Omega_{\alpha,M} \cap \{|t|> k\},\]
	and $k=k(\alpha,M,r)$ is that of Lemma \ref{lemmathin}.
	
	The integral over $\Omega_{\alpha,M}^1$ is finite for every $\ell$ since the integrand is bounded and $\Omega_{\alpha,M}^1$ has finite measure by \eqref{cylinder}. By Lemma \ref{lemmathin} and \eqref{cylinder} we get
	
	\[\int_{\Omega_{\alpha,M}^2 }\abs{\Omega_{\alpha,M} \cap B((x,t),r)}^\ell \leq C(\alpha,M,r) \int_{\{|x|\leq c\}} \int_{\{|t|>k\}}t^{\ell n\left(2-\alpha\right)}\,d t\,d x\]
	which is finite for $\ell > \frac{m}{n(\alpha -2)}$. This completes the proof.
\end{proof}
We can now prove Theorem~\ref{DSalpha}.

\begin{proof}[Proof of Theorem~\ref{DSalpha}]
	Let $\alpha \in (0,2]$ and let $(T_\alpha, \mathscr{D}_\alpha)$ be a self-adjoint extension of $\mathcal{L}^{w_\alpha}_0$. Then $(T_\alpha, \mathscr{D}_\alpha)$ is also a self-adjoint extension of $(\mathcal{L}_0^{w_\alpha}, C_c^\infty(G^*))$ and therefore $(U_\alpha T_\alpha U_\alpha^{-1}, 
U_\alpha\mathscr{D}_\alpha)$ is a self-adjoint extension of  $(\mathcal{L}+V_\alpha, C_c^\infty(G^*))$, since $U_\alpha \mathscr{D}_\alpha\supseteq C_c^\infty(G^*)$.  Since $U_\alpha T_\alpha U_\alpha^{-1}$ does not have purely discrete spectrum by Proposition~\ref{nonempty}, neither does $T_\alpha$.  
	
	\smallskip
	
	Let $\alpha \in \left( 2, \infty\right)$. By Lemma~\ref{casononcpt}, $V_\alpha$ is continuous and bounded from below. Therefore,  $\mathcal{L}^{w_\alpha}$ is unitarily equivalent to the closure of the operator $(\mathcal{L}+V_\alpha, C^\infty_c)$ by Proposition~\ref{uneq}. Since by Proposition~\ref{nonempty}  the closure of $(\mathcal{L}+V_\alpha, C^\infty_c)$ has purely discrete spectrum, the same holds for  $\mathcal{L}^{w_\alpha}$. 
\end{proof}

\begin{corollary}
	Let $\alpha>2$. Then, the semigroup $e^{-t\mathcal{L}^{w_\alpha}}$ is compact on $L^p(w_\alpha)$ for all $t>0$ and $1 < p < \infty$, and the spectrum of $\mathcal{L}^{w_\alpha}$ on $L^p(w_\alpha)$ is independent of $p$ for $1 < p < \infty$.
\end{corollary}
\begin{proof}
	The semigroup $(e^{-t\mathcal{L}^{w_\alpha}})_{t>0}$ is  Markovian, since it is associated with a Dirichlet form. Thus it is contractive on $L^\infty$ and positivity-preserving. Since the semigroup is symmetric and compact on $L^2(w)$ by  Theorem~\ref{DSalpha}, this proves the statement by~\cite[Theorem 1.6.3]{Davies1}.
\end{proof}

\smallskip

\subsection*{Acknowledgements} We are pleased to thank Giancarlo Mauceri for enlightening discussions, invaluable suggestions and a careful proofreading of a preliminary version of this paper.

\end{document}